\documentclass[10pt,leqno]{amsart}
\usepackage[UKenglish]{babel}
\usepackage{amsmath,amssymb,amsthm}
\usepackage[pdftex]{color}
\usepackage[bookmarks=true,hyperindex,pdftex,colorlinks, citecolor=MidnightBlue,linkcolor=MidnightBlue, urlcolor=MidnightBlue
]{hyperref}
\usepackage[dvipsnames]{xcolor}

\usepackage{tikz}
\usetikzlibrary{matrix}
\usetikzlibrary{arrows}
\tikzset{>=latex}

\usetikzlibrary{decorations.markings}
\tikzset{neg/.style={
            decoration={markings,
            mark= at position 0.5 with {
                  \node[transform shape] (tempnode) {$\setminus$};
                  }
              },
              postaction={decorate}
}}

\usepackage{enumerate}

\newtheorem{theorem}{Theorem}[section]
\newtheorem{lemma}[theorem]{Lemma}

\newtheorem{proposition}[theorem]{Proposition}
\newtheorem{corollary}[theorem]{Corollary}
\theoremstyle{definition}

\newtheorem{example}[theorem]{Example}

\theoremstyle{remark}
\newtheorem{remark}[theorem]{Remark}
\numberwithin{equation}{section}

\newcommand{\pten}{\ensuremath{\widehat{\otimes}_\pi}}

  \DeclareMathOperator{\diam}{diam\,}

  \newcommand{\ext}[1]{\operatorname{ext}\left(#1\right)}
\renewcommand{\leq}{\leqslant}
\renewcommand{\geq}{\geqslant}

\begin{document}



\title[SD2P and convex combinations of slices reaching the sphere]{Strong diameter two property and convex combinations of slices reaching the unit sphere}

\author[L\'{o}pez-P\'{e}rez]{Gin\'{e}s L\'{o}pez-P\'{e}rez}
\address[G.\ L\'opez-P\'erez]{Universidad de Granada, Facultad de Ciencias.
Departamento de An\'{a}lisis Matem\'{a}tico 18071-Granada
(Spain) and Instituto de Ma\-te\-m\'a\-ti\-cas de la Universidad de Granada (IEMath-GR)} \email{glopezp@ugr.es}
\urladdr{\url{http://wpd.ugr.es/local/glopezp/}}

\author[Mart\'in]{Miguel Mart\'in }
\address[M.\ Mart\'in]{Universidad de Granada, Facultad de Ciencias.
Departamento de An\'{a}lisis Matem\'{a}tico 18071-Granada
(Spain) and Instituto de Ma\-te\-m\'a\-ti\-cas de la Universidad de Granada (IEMath-GR)}
\email{mmartins@ugr.es}
\urladdr{\url{http://www.ugr.es/local/mmartins/}}

\author[Rueda Zoca]{Abraham Rueda Zoca }
\address[A.\ Rueda Zoca]{Universidad de Granada, Facultad de Ciencias.
Departamento de An\'{a}lisis Matem\'{a}tico, 18071-Granada
(Spain)} \email{ abrahamrueda@ugr.es}
\urladdr{\url{https://arzenglish.wordpress.com}}

\thanks{The research of Gin\'{e}s L\'{o}pez-P\'{e}rez and Miguel Mart\'{\i}n has been partially supported by Spanish MINECO/FEDER grant number MTM2015-65020-P, and Junta de Andaluc\'{\i}a/FEDER grant FQM-185. The research of Abraham Rueda Zoca has been supported  by MECD (Spain) FPU2016/00015, Spanish MINECO/FEDER grant number MTM2015-65020-P, and Junta de Andaluc\'{\i}a/FEDER grant FQM-185.}

\date{January 22nd, 2019}

\maketitle

\begin{abstract}
We characterise the class of those Banach spaces in which every convex combination of slices of the unit ball intersects the unit sphere as the class of those spaces in which every convex combination of slices of the unit ball contains two points at distance exactly two. Also, we study when the convex combinations of slices of the unit ball are relatively open or has non-empty relative interior for different topologies, studying the relationship between them and studying these properties for $L_{\infty}$-spaces and preduals of $L_1$-spaces.
\end{abstract}

\section{Introduction}

It is a well-known result in geometry of Banach spaces that every non-empty relatively weakly open subset of the unit ball contains a convex combination of slices of the unit ball (this result is sometimes known as Bourgain's lemma, cf.\  \cite[Lemma II.1]{ggms}, for instance). Although the reverse inclusion does not hold in general (cf.\ \cite[Remark IV.5]{ggms}), it may even happen for some Banach spaces that every convex combination of slices of the unit ball is relatively weakly open. The main result of \cite{al} shows that this is the case of $C(K)$ when the compact space $K$ is scattered. To study this phenomenon, the following properties were introduced in \cite[Section 3]{al}:
\begin{itemize}
\item [(W1)] Every convex combination of slices of the unit ball is weakly relatively open.
\item [(W2)] The relative weak interior of each convex combination of slices of the unit ball is not empty.
\item [(CS)] Every convex combination of slices of the unit ball intersects the unit sphere.
\end{itemize}
These properties had been already studied  implicitly in \cite{ggms}, as fundamental tools to the study of topological properties around the Radon-Nikod\'{y}m property in Banach spaces, as regularity and huskability.

Notice that (W1) implies (W2) which in turn implies (CS) for infinite-dimensional spaces. For finite-dimensional spaces, (CS) never happens while (W2) always does (see Proposition \ref{prop:inteconvexcom}). In \cite[Section 3]{al}, the authors wonder which class of spaces enjoy the above properties and if such spaces have any relation with the diameter two properties.

The main aim of this note is to clarify the relations between the above properties joint with similar properties in the setting of the norm topology and on the weak-star topology, and to show that there are  strong relations with the big slice phenomena, giving an affirmative answer to the question above. 

Before describing the content of the paper, let us introduce the analogous properties of (W1), (W2), and (CS) for the norm topology and the weak-star topology. Given a Banach space $X$, consider the following properties:

\begin{itemize}
    \item [(N1)] Every convex combination of slices of $B_X$ is relatively norm open.
    \item [(N2)] The relative norm interior of each convex combination of slices of $B_X$ is not empty.
\end{itemize}
Note that an analogous norm topology version of (CS) is the same than the weak version, as weak-open slices and norm-open slices are the same.

Additionally, if $X=Y^*$ is a dual Banach space, we define:
\begin{itemize}
    \item [(W$^*$1)] Every convex combination of weak$^*$-slices of $B_{Y^*}$ is relatively weakly-star open.
    \item [(W$^*$2)] The relative weakly-star interior of each convex combination of weak$^*$-slices of $B_{Y^*}$ is not empty.
    \item [(W$^*$-CS)] Every convex combination of weak$^*$-slices of $B_{Y^*}$ intersects $S_{Y^*}$.
\end{itemize}

We pass now to shortly describe the content of the manuscript.

In Section \ref{section:relations}, we study the properties (N1) and (N2) and their relations with the weak versions, clarifying the relation between all these properties. Among other results, we show that (N2) is satisfied by all Banach spaces and that strictly convex spaces always satisfy (N1) but always fail (CS).

The aim of Section \ref{section:relationwithSD2P} is to characterise the property (CS) in terms of a ``diameter two property'' kind condition, which gives solution to some questions in \cite{al}. Indeed, we show that a Banach space $X$ has the strong diameter two property (i.e.\ every convex combination of slices of the unit ball has diameter two) if, and only if, every convex combination of slices of the unit ball $C$ contains points arbitrarily close to the unit sphere of the space. The ideas involving the proof allow us to show that a Banach space $X$ enjoys the property (CS) if, and only if, every convex combination of slices of the unit ball has diameter two and the diameter is attained. We also give an example of a Banach space with the strong diameter two property but failing (CS). Besides, we show that the property (CS) is preserved by taking projective tensor product from both factors but not from only one of them.

Finally, we show in Section \ref{section:Linfty} that the properties (W$^*$1) and (W$^*$2) are equivalent for $L_{\infty}(\mu)$-spaces and that they are indeed equivalent to the fact that the localizable measure $\mu$ is purely atomic. We deduce that if a predual of a $L_1(\mu)$ space has (W2), then the measure $\mu$ has to be purely atomic.

\textbf{Notation:} We will only consider real Banach spaces. Given a Banach space $X$, we denote the closed unit ball (respectively the unit sphere) by $B_X$ (respectively $S_X$). We also denote by $X^*$ the topological dual of $X$. Given two Banach spaces $X$ and $Y$, $L(X,Y)$ stands for space of all bounded linear operators from $X$ to $Y$, and $X\pten Y$ is the projective tensor product of $X$ and $Y$ (see \cite{rya} for a detailed treatment of tensor products). Given a subset $C$ of $X$, $\ext{C}$ stands for the set of extreme points of $C$. By a \textit{slice} of $B_X$ we mean a set of the following form
$$S(B_X,f,\alpha):=\{x\in B_X\colon f(x)>1-\alpha\}$$
where $f\in S_{X^*}$ and $\alpha>0$. If $X=Y^*$ is a dual Banach space and $f$ actually belongs to the predual $Y$ of $X$, then the previous set is called a \textit{weak-star slice}. A \textit{convex combination of slices} of $B_X$ is a set of the following form
$$\sum_{i=1}^n \lambda_i S_i,$$
where $\lambda_1,\ldots, \lambda_n\in [0,1]$ satisfy that $\sum_{i=1}^n \lambda_i=1$ and each $S_i$ is a slice of $B_X$. In the case that $X$ is a dual space, we consider the analogous concept of \textit{convex combination of weak-star slices} of $B_X$.

A Banach space $X$ has the \textit{strong diameter two property} (\emph{SD2P} in short) if every convex combination of slices of the unit ball has diameter two. In the case that $X$ is a dual space, we say that $X$ has the  \textit{weak$^*$-strong diameter two property} (\emph{weak$^*$-SD2P} in short) if every convex combination of weak-star slices of $B_X$ has diameter two. We refer the reader to \cite{aln,blr,blr2} and references therein for background about diameter two properties.

\section{The relation between the norm and the weak topology versions}\label{section:relations}

The following is the general diagram of implications between the properties for the norm and for the weak topology for infinite-dimensional Banach spaces:
\begin{equation}\label{diageneral}
\begin{tikzpicture}[baseline=(m.center)]
  \matrix (m) [matrix of math nodes,row sep=3em,column sep=4em,minimum width=2em]
  {
     \text{(W1)} & \text{(W2)} & \text{(CS)} \\
     \text{(N1)} & \text{(N2)} &  \\};
  \path[-stealth]
    (m-1-1) edge [double] node [right] {(3)} (m-2-1)
            edge [double] node [above] {(1)} (m-1-2)
    (m-1-2) edge [double] node [above] {(2)} (m-1-3)
            edge [double] node [right] {(4)} (m-2-2)
    (m-2-1) edge [double] node [above] {(5)} (m-2-2);
\end{tikzpicture}
\end{equation}

Let us show that none of the reverse implications hold. Indeed, the fact that the reverse implications of (1) and (2) do not hold was proved in \cite[Corollary 2.5 and Corollary 2.9]{hkp} (a counterexample for (1) is $c_0\oplus_1 c_0$ whereas one for (2) is $c_0\oplus_\infty\ell_2$). In order to prove the corresponding statements for the implications (3), (4), and (5), let us begin with the following proposition, from which an easy consequence is that every Banach space satisfies (N2).

\begin{proposition}
\label{prop:inteconvexcom}
Let $X$ be a Banach space and let $C:=\sum_{i=1}^n \lambda_i S_i$ be a convex combination of slices of $B_X$. Then
$$C\cap \operatorname{int}(B_X)=\sum_{i=1}^n \lambda_i (S_i\cap \operatorname{int}(B_X)).$$
In particular, every point of $C\cap \operatorname{int}(B_X)$ is norm-interior to $C$.
\end{proposition}

\begin{proof}
The inclusion $\sum_{i=1}^n \lambda_i (S_i\cap \operatorname{int}(B_X))\subseteq C\cap \operatorname{int}(B_X)$ is clear from an easy convexity argument. In order to prove the reverse inclusion, let $x\in C\cap \operatorname{int}(B_X)$, so $x=\sum_{i=1}^n\lambda_i x_i$ for suitable $x_i\in S_i$ for every $i\in\{1,\ldots, n\}$. Since each $S_i$ is a relative norm-open subset of $B_X$, we can find $\varepsilon>0$ small enough so that $B(x_i,2\varepsilon)\cap B_X\subseteq S_i$ holds for every $i\in\{1,\ldots, n\}$. Define
$$z_i:=(1-\varepsilon)x_i+\varepsilon x,$$
which satisfies that $\sum_{i=1}^n\lambda_i z_i=x$. It remains to prove that, given $i\in\{1,\ldots, n\}$, $z_i\in S_i\cap \operatorname{int}(B_X)$, for which we will prove that $z_i\in B(x_i,2\varepsilon)\cap \operatorname{int}(B_X)$. Given $i\in\{1,\ldots, n\}$, we get that
$$\Vert z_i-x_i\Vert=\Vert \varepsilon(x_i+x)\Vert\leq \varepsilon\Vert x_i+x\Vert\leq 2\varepsilon,$$
which proves that $z_i\in B(x_i,2\varepsilon)$. Moreover,
$$\Vert z_i\Vert\leq (1-\varepsilon)\Vert x_i\Vert+\varepsilon\Vert x\Vert\leq (1-\varepsilon)+\varepsilon\Vert x\Vert<(1-\varepsilon)+\varepsilon=1,$$
where the last inequality is strict because $\Vert x\Vert<1$ by assumption. This proves that $z_i\in \operatorname{int}(B_X)$, which finishes the proof.\end{proof}

An inmediate consequence of the previous proposition is the following corollary.

\begin{corollary}\label{coro:todoN2}
Every Banach space $X$ has the property (N2).
\end{corollary}

In view of the previous corollary, every Banach space $X$ failing (W2) (e.g.\ $C[0,1]$ by \cite[Theorem 3.1]{hkp}) proves that the converse of (4) does not hold.

For the converse of (3), the following proposition provides a large class of counterexamples.

\begin{proposition}\label{prop:strictlyconvex}
Let $X$ be a strictly convex Banach space. Then $X$ satisfies (N1) but fails (CS).
\end{proposition}

\begin{proof} Let us begin by proving that $X$ fails (CS). Consider two disjoint slices $S_1,S_2$ of $B_X$ and $C:=\frac{S_1+S_2}{2}$, and we claim that $C\cap S_X=\emptyset$. Indeed, if there exist $z\in C\cap S_X$, then there exist $x\in S_1$, $y\in S_2$ such that $z=\frac{x+y}{2}$. Since $z\in S_X$ is an extreme point, then $x=y=z$, which is impossible because $S_1$ and $S_2$ were taken to be disjoint. This proves that $C\cap S_X=\emptyset$.

In order to prove that $X$ satisfies (N1), pick a convex combination of slices $C:=\sum_{i=1}^n\lambda_i S_i$ of $B_X$ and $x\in C$, and let us prove that $x$ is an interior point of $C$. Now, we have two possibilities:

\noindent (a). If $\Vert x\Vert<1$, then $x$ is a norm interior point of $C$ by Proposition \ref{prop:inteconvexcom}.

\noindent (b). If $\Vert x\Vert=1$ then, since $x$ is an extreme point, we conclude as before that $x\in \bigcap\limits_{i=1}^n S_i\subseteq C$. This again proves that $x$ is a norm-interior point, since $\bigcap\limits_{i=1}^n S_i$ is a relative norm-open set.
\end{proof}

\begin{remark}
Note that the the first part of the previous proof already appeared in an (unpublished) previous version of this note  \href{https://arxiv.org/abs/1703.04749v1}{arXiv:1703.04749v1}, publicly available at March 2017.
\end{remark}

In order to prove that the converse of (5) in \eqref{diageneral} does not hold, note that an easier reformulation of \cite[Proposition 3.3 (b)]{abhlp} is that if $\ext{B_X}$ is not norm-closed then $X$ fails (W1). The following proposition shows that much more can be said.

\begin{proposition}\label{prop:extre}
Let $X$ be a Banach space. Then:
\begin{itemize}
    \item[(1)] If $X$ has (W1), then $\ext{B_X}$ is weakly closed in $B_X$.
    \item[(2)] If $X$ is infinite-dimensional and has (W2), then $\ext{B_X}$ can not be weakly-dense.
\end{itemize}
\end{proposition}

\begin{proof}
In order to prove (1), consider a net $\{x_s\}$ of extreme points which is weakly convergent to some $x\in B_X$. We claim that $x$ is an extreme point of $B_X$. In fact, assume by contradiction the existence of a pair of points $y,z\in B_X$ such that $x=\frac{y+z}{2}$. By the Hahn-Banach theorem we can find a pair of slices $S_1, S_2$ of $B_X$ satisfying that $y\in S_1, z\in S_2$ and $S_1\cap S_2=\emptyset$. Since $C=\frac{S_1+S_2}{2}$ is weakly open, we can find an index $s$ such that $x_s\in \frac{S_1+S_2}{2}$. Since the slices $S_1$ and $S_2$ are disjoint, there are two different elements $y_s\in S_1, z_s\in S_2$ such that $x_s=\frac{y_s+z_s}{2}$, getting a contradiction with the fact that $x_s$ is an extreme point. Consequently, $x\in \ext{B_X}$, as desired.

For the proof of (2), notice that Proposition \ref{prop:strictlyconvex} implies that $X$ is not stricly convex, so there exists $z\in S_X$ which is not an extreme point. Now, an adaptation of the proof of (1) does the trick.
\end{proof}

Note that similar arguments allow us to derive analogous consequences for the rest of properties.

\begin{proposition}
Let $X$ be a Banach space. Then,
\begin{enumerate}
\item If $X$ has (N1) then $\ext{B_X}$ is norm closed.
\end{enumerate}

Moreover, if $X$ is a dual Banach space, then:
\begin{enumerate}
\item[(2)] If $X$ has (W$^*$1) then $\ext{B_X}$ is weakly-star closed.
\item[(3)] If $X$ has (W$^*$2) then $\ext{B_X}$ is not weakly-star dense in $B_X$.
\item[(4)] If $X$ has (W$^*$-CS) then $X$ is not striclty convex.
\end{enumerate}
\end{proposition}

\begin{example}
Consider $X=C[0,1]$. It is obvious that $\ext{B_X}=\{\pm \mathbf{1}\}$ is norm-compact, but $X$ fails (W2) by \cite[Theorem 3.1]{hkp}. This shows that the converse of (1) and (2) in Lemma \ref{prop:extre} do not hold.
\end{example}

It is well known that in every Banach space $X$ with $\dim(X)\geq 3$, there exists a closed, convex and bounded subsets with a non-empty interior $C$ so that $\ext{B_X}$ is not closed. Since such $C$ can be seen as an equivalent unit ball in the space $X$, we get the following corollary.

\begin{corollary}
Let $X$ be a Banach space such that $\dim(X)\geq 3$. Then there exists an equivalent norm on $X$ failing the property (N1) (and thus failing (W1)).
\end{corollary}

In particular, the previous corollary exhibit a large class of examples which show that the reverse of (5) in (\ref{diageneral}) does not hold.

\section{Characterisation of (CS) and interrelation with the SD2P}\label{section:relationwithSD2P}

In \cite[Section 3]{al} it is stated to be unclear whether there is any connection between having weakly open convex combination of slices and the diameter two properties. The following argument shows that the strong diameter two property is a necessary condition.

\begin{theorem}\label{carasd2p}
Let $X$ be a Banach space. The following assertions are equivalent:
\begin{itemize}
\item[(1)] $X$ has the strong diameter two property.
\item[(2)] For every convex combination of slices $C$ of $B_X$ and every $\varepsilon>0$, there exists $x\in C$ such that $\Vert x\Vert>1-\varepsilon$.
\end{itemize}
\end{theorem}

\begin{proof}
(1)$\Rightarrow$(2) is obvious, so let us prove (2)$\Rightarrow$(1). To this end, pick a convex combination of slices $C:=\sum_{i=1}^n S(B_X,f_i,\alpha)$ of $B_X$ and $\varepsilon>0$. Define $$D:=\frac{1}{2}\left(\sum_{i=1}^n \lambda_i S(B_X,f_i,\alpha)+\sum_{i=1}^n \lambda_i S(B_X,-f_i,\alpha) \right),$$ which is also a convex combination of slices of $B_X$. Choose $x=\frac{1}{2}\left(\sum_{i=1}^n \lambda_i x_i+\sum_{i=1}^n \lambda_i y_i\right)\in D$ with $\Vert x\Vert>1-\varepsilon$. Notice that, by the definition of $D$, we get that $-\sum_{i=1}^n\lambda_i y_i,\sum_{i=1}^n\lambda_i x_i\in C$. Consequently,
$$
\diam(C)\geq \left\Vert \sum_{i=1}^n \lambda_i x_i-\left(-\sum_{i=1}^n \lambda_i y_i\right)\right\Vert=2\Vert x\Vert>2(1-\varepsilon).
$$
Since $\varepsilon>0$ is arbitrary, we get that $\diam(C)=2$.
\end{proof}

Note that the same proof gives a weak-star version of the previous theorem.

\begin{proposition}
Let $X$ be a Banach space. The following assertions are equivalent:
\begin{itemize}
\item[(1)] $X^*$ has the weak$^*$-strong diameter two property.
\item[(2)] For every convex combination of weak$^*$-slices $C$ of $B_{X^*}$ and every $\varepsilon>0$, there exists $x^*\in C$ such that $\Vert x^*\Vert>1-\varepsilon$.
\end{itemize}
\end{proposition}

Theorem \ref{carasd2p} shows that the property (CS) implies the SD2P. The converse, however, is not longer true.

\begin{example}\label{cor}
There exist Banach spaces $X$ with the SD2P failing (CS).
\end{example}

\begin{proof}
An example of a strictly convex space being a non-reflexive M-embedded Banach space (and hence with the SD2P by \cite[Theorem 4.10]{aln}) $X$ is exhibited in \cite[p.~168]{hww}. From Proposition \ref{prop:strictlyconvex}, this Banach space fails (CS).
\end{proof}

In \cite[Question (iii)]{al} it is asked which Banach spaces verify (CS). A slight modification in the proof of Theorem \ref{carasd2p} yields a characterisation of those spaces in terms of the diameter of convex combination of slices.

\begin{theorem}\label{carapropi3}
Let $X$ be a Banach space. The following are equivalent:
\begin{itemize}
\item[(1)] $X$ satisfies the property (CS).
\item[(2)] For every convex combination of slices $C$ of $B_X$ there are $x,y\in X$ such that $\Vert x-y\Vert=2$.
\end{itemize}
\end{theorem}

\begin{proof}
(2) implies (1) is clear. For (1) implies (2), consider a convex combination of slices of $B_X$ given by $C:=\sum_{i=1}^n S(B_X,f_i,\alpha)$.   Define $$D:=\frac{1}{2}\left(\sum_{i=1}^n \lambda_i S(B_X,f_i,\alpha)+\sum_{i=1}^n \lambda_i S(B_X,-f_i,\alpha \right)),$$ which is also a convex combination of slices of $B_X$. Choose, from the assumption, $$x_0=\frac{1}{2}\left(\sum_{i=1}^n \lambda_i x_i+\sum_{i=1}^n \lambda_i y_i\right)\in D\cap S_X .$$ Now $x:=\sum_{i=1}^n \lambda_i x_i\in C$, $y:=-\sum_{i=1}^n \lambda_i y_i\in C$ and $\|x-y\|=2\Vert x_0\Vert=2$.
\end{proof}

As well as happen with Theorem \ref{carasd2p}, an analogous statement to the previous theorem can be stated for (W$^*$-CS).

\begin{proposition}
Let $X$ be a Banach space. The following are equivalent:
\begin{itemize}
\item[(1)] $X^*$ satisfies the property (W$^*$-CS).
\item[(2)] For every convex combination of weak-star slices $C$ of $B_X$ there are $x^*,y^*\in C$ satisfying that $\Vert x^*-y^*\Vert=2$.
\end{itemize}
\end{proposition}

Let us conclude with some consequences related to preservance of the property (CS) by taking projective tensor products. The next proposition follows similar ideas to the ones of \cite[Theorem 3.5]{blr}.

\begin{proposition}
Let $X$ and $Y$ be two Banach spaces with the property (CS). Then the space $X\pten Y$ also satisfies (CS).
\end{proposition}

\begin{proof}
Consider $C:=\sum_{i=1}^n S(B_{X\pten Y}, T_i,\alpha)$ to be a convex combination of slices of $B_{X\pten Y}$, where $T_i\in (X\pten Y)^*\equiv L(X,Y^*)$ (we refer to \cite[Chapter 2]{rya}), and let us prove that $C\cap S_{X\pten Y}\neq \emptyset$. Indeed, consider $u_i\otimes v_i\in S(B_{X\pten Y}, T_i,\alpha)\cap (S_X\otimes S_Y)$ for all $i\in\{1,\ldots, n\}$. Now
$$u_i\otimes v_i\in S(B_{X\pten Y}, T_i,\alpha)\Leftrightarrow T_i(u_i)(v_i)>1-\alpha\Leftrightarrow u_i\in S(B_X,v_i\circ T_i,\alpha).$$
By assumption there exists an element $\sum_{i=1}^n \lambda_i x_i\in \sum_{i=1}^n \lambda_i S(B_X,v_i\circ T_i,\alpha)$ whose norm is 1. By the Hahn-Banach theorem we can find a functional $x^*\in S_{X^*}$ such that $x^*(x_i)=1$ holds for all $i\in\{1,\ldots, n\}$. It is obvious that $\sum_{i=1}^n \lambda_i x_i\otimes v_i\in C$. Now, by the same procedure we get elements $y_1,\ldots, y_n\in B_Y$ and a functional $y^*\in S_{Y^*}$ such that $y^*(y_i)=1$ holds for every $i\in\{1,\ldots, n\}$ and such that $\sum_{i=1}^n \lambda_i x_i\otimes y_i\in C$. Now
$$\left\Vert \sum_{i=1}^n \lambda_i x_i\otimes y_i\right\Vert\geq \sum_{i=1}^n \lambda_i x^*(x_i)y^*(y_i)=1.$$
Consequently, $C\cap S_{X\pten Y}\neq \emptyset$ as desired.
\end{proof}

\begin{remark}
The assumption of the property on both factors is necessary. In fact, consider $X=\ell_\infty$ and $Y=\ell_p^3$ for some $2<p<\infty$. Note that every convex combination of slices of $B_X$ intersects the unit sphere \cite[Example 3.3]{al}. However, this is not longer true for $X\pten Y$ because such space even fails the strong diameter two property \cite[Corollary 3.9]{llr}, so Theorem \ref{carasd2p} yields the existence of a convex combination of slices $C$ in $B_{X\pten Y}$ and a radius $0<r<1$ such that $C\subseteq r B_{X\pten Y}$.
\end{remark}

\section{The weak-star properties for $L_\infty(\mu)$-spaces}\label{section:Linfty}

Note that \cite[Theorem 3.1]{hkp} proves that, given a compact Hausdorff topological space $K$, then if $C(K)$ has the property (W2) then $K$ admits an atomeless measure. Our aim is to generalise this result to the context of $L_1$-preduals. In order to do so, we will analyse the properties (W$^*$1) and (W$^*$2) in $L_\infty(\mu)$ spaces. More precisely, let $(\Omega,\Sigma,\mu)$ be a localizable measure space. We wonder when $L_\infty(\mu)=L_1(\mu)^*$ satisfies that every convex combination of weak-star slices of $B_{L_\infty(\mu)}$ is a weak-star open subset of $B_{L_\infty(\mu)}$. Let us state the following result, which gives a complete answer to the previous question.

\begin{theorem}\label{teo:caraW*2linfinito}
Let $(\Omega,\Sigma,\mu)$ be a localizable measure space. The following assertions are equivalent:
\begin{enumerate}
\item\label{teo:caraW*2linfinito1} $L_\infty(\mu)$ has (W$^*$1).
\item\label{teo:caraW*2linfinito2} $L_\infty(\mu)$ has (W$^*$2).
\item\label{teo:caraW*2linfinito3} $\mu$ is purely atomic.
\end{enumerate}
\end{theorem}

In order to prove Theorem \ref{teo:caraW*2linfinito} we will need several preliminary results. We will start with a pair of results which will result in the proof of \eqref{teo:caraW*2linfinito2}$\Rightarrow$\eqref{teo:caraW*2linfinito3} in Theorem \ref{teo:caraW*2linfinito}.

\begin{lemma}\label{lemma:condineceLinfty} Let $(\Omega,\Sigma,\mu)$ be a finite measure space. If $\mu$ does not contain any atom, then there exists a convex combination of weak-star slices $C$ of $B_{L_\infty(\mu)}$ which does not contain any weak-star interior point. In other words, if $\mu$ is not purely atomic, then $L_\infty(\mu)$ fails (W$^*$2).
\end{lemma}

\begin{proof}
The proof is an adaptation of that of \cite[Theorem 3.1]{hkp}. We will assume with no loss of generality that $\mu(\Omega)=1$. Since $\mu$ does not contain any atom then we can find three disjoint measurable sets $A,B,C\in \Sigma$ such that $A\cup B\cup C=\Omega$ and such that $\mu(A)=\mu(B)=\mu(C)=\frac{1}{3}$. Using the previous sets we define the following functions
\[\begin{split}
f_1:=\chi_A+\chi_B-\chi_C,\  f_2:=\chi_A-\chi_B-\chi_C.
\end{split}
\]
It is clear that $f_1,f_2\in L_1(\mu)$ are one-norm functions. Pick $0<\varepsilon<\frac{1}{12}$ and define
$$S_1=S(B_{L_\infty(\mu)},f_1,\varepsilon^2),\qquad S_2=S(B_{L_\infty(\mu)},f_2,\varepsilon^2).$$
Define $C:=\frac{S_1+S_2}{2}$. We will prove that $C$ does not have interior points. To this end, we start by giving a necessary condition for an element of $B_{L_\infty(\mu)}$ to belong to $C$. For this we introduce a bit of notation. For a function $u\in B_{L_\infty(\mu)}$, we define the following sets:
$$\begin{array}{ccc}
B_1^u:=\{t\in B:u(t)\leq 1-\varepsilon\}, & B_{-1}^u:=\{t\in B: u(t)\geq -1+\varepsilon\}, &
 B_0^u:=\{t\in B: \vert u(t)\vert\geq \varepsilon\}.
\end{array}$$

\emph{Claim}. If $u\in C$, then $\mu(B_0^u)\leq 2\varepsilon$.\newline
\noindent Indeed, given $u\in C$ then $u=\frac{x+y}{2}$ for suitable $x\in S_1$ and $y\in S_2$. We claim that $\mu(B_1^x)<\varepsilon$. Otherwise if $\mu(B_1^x)\geq \varepsilon$ then we would get
$$1-\varepsilon^2<x(f_1)=\int_\Omega xf_1\ d\mu=\int_A xd\mu+\int_{B_1^x}x d\mu+\int_{B\setminus B_1^x}xd\mu+\int_C xd\mu.$$
Notice that $x(t)\leq 1$ whenever $t\in A\cup (B\setminus B_1^x)\cup C$, whereas $x(t)\leq 1-\varepsilon$ if $t\in B_1^x$. Consequently, the following equalities hold
\[\begin{split}
x(f_1)&=\mu(A)+\mu(B\setminus B_1^x)+\mu(C)+(1-\varepsilon)\mu(B_1^x)\\
& =1-\mu(B_1^x)+(1-\varepsilon)\mu(B_1^x)=1-\varepsilon\mu(B_1^x).
\end{split}
\]
Since $\mu(B_1^x)\geq\varepsilon$ we get that $x(f_1)<1-\varepsilon^2$, which entails a contradiction with the assumption that $x\in S_1$. Consequently $\mu(B_1^x)<\varepsilon$ as desired. Similar computations also proves that $\mu(B_{-1}^y)<\varepsilon$. Moreover, notice that $(B\setminus B_1^x)\cap (B\setminus B_{-1}^y)\subseteq B\setminus B_0^u$ or, equivalently, $B_0^u\subseteq B_1^x\cup B_{-1}^y$. From here the claim easily follows.

Now, using the previous claim we will prove that $C$ does not have any weak-star interior point. Pick $z\in C$, consider a weak-star neighbourhood $\mathcal U$ of $z$ and let us find an element $u\in \mathcal U\setminus C$. Since $\mathcal U$ is weak-star open, we can assume that $\mathcal U$ is of the form
$$
\mathcal U=\bigl\{u\in B_{L_\infty(\mu)}\colon \vert (u-z)(\varphi_i)\vert<\gamma,\ i=1,\ldots, n\bigr\}
$$
for suitable $n\in\mathbb N, \gamma>0$ and $\varphi_1,\ldots, \varphi_n\in S_{L_1(\mu)}$. In order to find an element $u\in \mathcal U\setminus C$, define the sets
$$
E:=A\cup C\cup B_0^z \qquad \text{and} \qquad D:=B\setminus B_0^z=B\setminus E=\Omega\setminus E.
$$
By \cite[Lemma 3.2]{hkp} and by using an application of Hahn decomposition theorem similar to the one of the proof of \cite[Theorem 3.1]{hkp}, we can find two disjoint sets $D_1,D_2\in \Sigma$ such that $D_1\cup D_2=D$ and such that
\begin{equation}\label{desivariamedi}
\left\vert \int_{D_1}f_id\mu-\int_{D_2}f_i d\mu\right\vert<\delta\ \forall i\in\{1,\ldots, n\},
\end{equation}
for $0<\delta<\min\left\{\frac{\gamma}{3(1-\varepsilon)},\frac{1}{6}-2\varepsilon\right\}$. Note that we can assume that $\mu(D_1)>0$ and $\mu(D_2)>0$. Moreover, we can find two sets $\widehat{D_1}\subseteq D_1$ and $\widehat{D_2}\subseteq D_2$ such that $0<\mu(\widehat{D_i})<\delta$ for $i=1,2$. Finally, define $u$ as follows
$$u(t):=\left\{\begin{array}{cc}
z(t)+1-\varepsilon & \mbox{if }t\in D_1\setminus \widehat{D_1},\\
z(t)-1+\varepsilon & \mbox{if }t\in D_2\setminus \widehat{D_2},\\
z(t)& \mbox{otherwise}.
\end{array} \right.$$
Finally, let us show that $u\in \mathcal U\setminus C$. It is clear that $u\in B_{L_\infty(\mu)}$ since $D=B\setminus B_0^z=\{t\in B:\vert z(t)\vert<\varepsilon\}$. Let us prove that $u\in \mathcal{U}$. To this end, fix $i\in\{1,\ldots, n\}$. Then
$$\varphi_i(u-z)=\int_{E\setminus ((D_1\setminus \widehat{D_1})\cup(D_2\setminus \widehat{D_2}))}(u-z)d\mu+\int_{D_1\setminus \widehat{D_1}}(u-z) d\mu+\int_{D_2\setminus \widehat{D_2}}(u-z) d\mu.$$
Note that the first integral is 0 because $u=z$ on the integrating. On the other hand, $u-z\leq 1-\varepsilon$ on $D_1\setminus \widehat{D_1}$ as well as $u-z\geq -1+\varepsilon$ on $D_2\setminus \widehat{D_2}$. Consequently, the remaining two summands can be estimated as follows
\[\begin{split}
\varphi_i(u-z)\leq & (1-\varepsilon)\left(\int_{D_1\setminus \widehat{D_1}}f_i d\mu-\int_{D_2\setminus \widehat{D_2}}f_i d\mu\right)\\
& \leq (1-\varepsilon)\left(\int_{D_1}f_id\mu-\int_{D_2}f_i d\mu-\mu(\widehat{D_1})-\mu(\widehat{D_2})\right)\\
& < 3(1-\varepsilon)\delta<\gamma.
\end{split}\]
Therefore, $u\in \mathcal U$. In order to prove that $u\notin C$, pick $t\in (D_1\setminus \widehat{D_1})\cup (D_2\setminus \widehat{D_2})$ and notice that
$$\vert u(t)\vert\geq 1-\varepsilon-\vert z(t)\vert>2-2\varepsilon>\varepsilon,$$
so $t\in B_0^u$, which proves that $(D_1\setminus \widehat{D_1})\cup (D_2\setminus \widehat{D_2})\subseteq B_0^u$. Consequently, we get
\[\begin{split}
\mu(B_0^u)\geq \mu(D_1)+\mu(D_2)-\mu(\widehat{D_1})-\mu(\widehat{D_2})\geq& \mu(D)-2\delta=\mu(B)-\mu(B_0^z)-2\delta\\
& >\frac{1}{3}-2\varepsilon-2\delta>2\varepsilon
\end{split}
\]
where we have used that $\mu(B_0^z)<2\varepsilon$ since $z\in C$. Consequently $\mu(B_0^u)\geq 2\varepsilon$ and, according to the claim, $u$ does not belong to $C$ as desired.
\end{proof}

Our aim is now to remove the finiteness assumption from the previous lemma. In order to do so, we need the following proposition, which can be seen as a weak-star version of \cite[Proposition 2.7]{hkp}.

\begin{lemma}\label{lemma:sumainfi}
Let $X$ and $Y$ be two Banach spaces and let $Z:=X\oplus_1 Y$. If $Z^*=X^*\oplus_\infty Y^*$ has (W$^*$2), then $X^*$ and $Y^*$ have (W$^*$2).
\end{lemma}

\begin{proof}
The proof will be an adaptation of that of \cite[Proposition 2.7]{hkp}. We will only prove that $X^*$ has (W$^*$2). Let $C:=\sum_{i=1}^n \lambda_i S(B_{X^*}, x_i,\alpha_i)$ be a convex combination of $w^*$-slices of $B_{X^*}$ and let $x^*\in C$. Define
$$D:=\sum_{i=1}^n \lambda_i S(B_{Z^*},(x_i,0),\alpha_i),$$
which is clearly a convex combination of $w^*$-slices of $B_{Z^*}$. Moroever, it is clear that $(z,0)\in D$. Since $Z^*$ has (W$^*$2), it follows that there exists a weak-star open subset $W$ of $B_{Z^*}$ such that $(z,0)\in W\subseteq D$. Since finite-intersections of weak-star slices are basis of the weak-star topology of $B_{Z^*}$ we can assume, with no loss of generality, that
$$W=\bigcap\limits_{i=1}^k S(B_{Z^*}, (a_i,b_i),\beta_i)$$
for suitable $k\in \mathbb N$, $a_i\in X, b_i\in Y$ such that $\Vert a_i\Vert+\Vert b_i\Vert=1$ and $\beta_i>0$ for every $i\in\{1,\ldots, k\}$.
Since $(x^*,0)\in W$ it follows that, given $i\in\{1,\ldots, k\}$, then $1-\beta_i< z^*(a_i)=(z^*,0)(a_i,b_i)\leq \Vert a_i\Vert$. Now, define
$$U:=\bigcap\limits_{i=1}^k \{f\in B_{X^*}: f(a_i)>1-\beta_i\}.$$
It is clear that $U$ is a weak-star open subset of $B_{X^*}$ and that $x^*\in U$. In order to finish the proof let us prove that $U\subseteq C$. To this end, choose $u^*\in U$. From the definition of $U$ and $W$ it follows that $(u^*,0)\in W$. Since $W\subseteq D$ then we can find, for every $i\in\{1,\ldots, n\}$, an element $(a_i^*,b_i^*)\in S(B_{Z^*},(x_i,0),\alpha_i)$ such that
$$(u,0)=\sum_{i=1}^n \lambda_i (a_i,b_i).$$
This means that $u^*=\sum_{i=1}^n \lambda_i a_i^*$. Furthermore, because of the definition of the norm on $Z^*$, it follows that $\Vert a_i^*\Vert\leq 1$. Finally, given $i\in\{1,\ldots, n\}$, we get
$$a_i^*(x_i)=(a_i^*,b_i^*)(x_i,0)>1-\beta$$
because, by assumptions, $(a_i^*,b_i^*)\in S(B_{Z^*},(x_i,0),\alpha_i)$. This proves that $u^*=\sum_{i=1}^n \lambda_i a_i^*\in C$, which in turn implies that $U\subseteq C$ and finishes the proof.
\end{proof}

Now, we are ready to prove the following result.

\begin{proposition}\label{propoLinftynow*2}
Let $(\Omega,\sigma,\mu)$ be a localizable measure space. If $\mu$ is not purely atomic, then $L_\infty(\mu)$ fails the property (W$^*$2).
\end{proposition}

\begin{proof}
Since $\mu$ is not purely atomic, we can find a measurable subset $A\subseteq \Omega$ such that $0<\mu(A)<\infty$ so that $\mu_{|A}$ is a non-atomic measure. Notice that $L_1(\mu)=L_1(\mu_{|A})\oplus_1 L_1(\mu_{|\Omega\setminus A})$ (via the surjective linear isometry $f\longmapsto (f \chi_A,f\chi_{\Omega\setminus A})$). This raises the following decomposition
$$L_\infty(\mu)=L_\infty(\mu_{|A})\oplus_\infty L_\infty(\mu_{|\Omega\setminus A}).$$
Since $\mu_{|A}$ is a finite non-atomic measure, Lemma \ref{lemma:condineceLinfty} implies that $L_\infty(\mu_{|A})$ fails the property (W$^*$2), so $L_\infty(\mu)$ fails the property (W$^*$2) by Lemma \ref{lemma:sumainfi}, as desired.
\end{proof}

In the purely atomic case, the conclusions are dramatically different. The proof of the next result is an adaptation of that of \cite[Theorem 2.3]{al}.

\begin{proposition}\label{prop:condisufilinfty}
Let $I$ be a non-empty set. Then every convex combination of weak-star slices of $B_{\ell_\infty(I)}$ is relatively weak$^*$-open. In other words, $\ell_\infty(I)$ has property (W$^*$1).
\end{proposition}

\begin{proof}  Consider $C:=\sum_{i=1}^n\lambda_i S(B_{\ell_\infty(I)},f_i,\alpha)$, pick $z=\sum_{i=1}^n \lambda_i x_i\in C$ and consider $\delta>0$ such that $$\langle x_i,f_i\rangle>1-\alpha+\delta\bigl(\min\limits_{1\leq j\leq n}\lambda_j\bigr)^{-1}.$$ Since $f_i\in \ell_1(I)$, we can find a finite set $F\subseteq I$ such that $\sum_{t\in I\setminus F}\vert f_i(t)\vert<\frac{\delta}{3}$ for every $i\in\{1,\ldots, n\}$. Moreover, define $$q:=\min\{1-\vert x_i(t)\vert\colon \vert x_i(t)\vert<1,\, t\in F,\, i=1,\ldots, n\}$$ and define
$$\mathcal U:=\{y\in B_{\ell_\infty(I)}\colon \vert y(t)-z(t)\vert<\varepsilon,\, t\in F\},$$
where $\varepsilon<q\left(\min\limits_{1\leq j\leq n}\lambda_j\right)^{-1}$ if $q\neq 0$ and $\varepsilon \left(\min\limits_{1\leq j\leq n}\lambda_j\right)^{-1}<\frac{\delta}{3}$.

It is obvious that $z\in\mathcal U$. In order to prove that $\mathcal U\subseteq C$, consider $y\in\mathcal U$. Our aim is to write $y=\sum_{i=1}^n \lambda_i y_i$ for suitable $y_i\in S(B_{\ell_\infty(I)},f_i,\alpha)$, for which we will follow word-by-word the proof of \cite[Theorem 2.3]{al}. To this end, we will define $y_i$ by coordinates. Pick $t\in I$ and let us discuss by cases:\newline
\noindent (1). If $t\in I\setminus F$, we simply define $y_i(t)=y(t)$ for every $i\in\{1,\ldots, n\}$.\newline
\noindent (2). If $t\in F$ and there exists $i_0\in \{1,\ldots, n\}$ such that $\vert x_{i_0}(t)\vert<1$, define $y_{i_0}(t)=x_{i_0}(t)+\frac{y(t)-z(t)}{\lambda_{i_0}}$ and $y_i(t)=x_i(t)$ for $i\neq i_0$. Note that
$$\sum_{i=1}^n\lambda_i y_i=\sum_{i=1}^n \lambda_i x_i(t)-z(t)+y(t)=y(t).$$
Moreover, because of the choice of $\varepsilon$ in that case, we have that
$$\vert y_{i_0}(t)\vert\leq \vert x_{i_0}(t)\vert+\frac{\varepsilon}{\lambda_{i_0}}<\vert x_{i_0}\vert+1-\vert x_{i_0}\vert=1,$$
so it is clear that $\vert y_i(t)\vert\leq 1 $ holds for every $i\in\{1,\ldots, n\}$. Notice also that
$$\vert y_i(t)-x_i(t)\vert\leq \frac{\varepsilon}{\min\limits_{1\leq j\leq n}\lambda_j}<\frac{\delta}{3}.$$
\noindent (3). If $t\in F$, $\vert x_i(t)\vert=1$, and all the $x_i(t)$ are equal, then one defines $y_i(t)=y(t)$ since, in this case, $x_i(t)=x(t)$ and so $\vert x_i(t)-y_i(t)\vert=0<\frac{\delta}{3}$.\newline
\noindent (4). Finally, if $t\in F$, $\vert x_i(t)\vert=1$ but not all $x_i(t)$ are equal, we define the following sets:
$$A:=\{i\in\{1,\ldots, n\}\colon x_i(t)=1\}\ \mbox{ and }\ B:=\{i\in\{1,\ldots, n\}\colon x_i(t)=-1\}.$$
Note that, by assumptions, $A\cup B=\{1,\ldots, n\}$. Define also $\Lambda_A:=\sum_{i\in A}\lambda_i$ and $\Lambda_B:=\sum_{i\in B}\lambda_i$ and note that $\Lambda_A+\Lambda_B=1$. In order to save notation, for an index $i\in\{1,\ldots, n\}$, define $A_i=A$ if $i\in A$ and $A_i=B$ if $i\in B$. Now, for every $i\in\{1,\ldots, n\}$, we are ready to define $y_i(t)$ as follows
$$y_i(t):=x_i(t)-\delta\frac{\operatorname{sign}(x_i(t))}{2\Lambda_{A_i}}+\frac{(y-x)(t)}{2\Lambda_{A_i}}.$$
Notice that in $\vert y_i(t)\vert\leq 1$. Indeed, the following inequality is clear
$$\vert y_i(t)\vert\leq \vert x_i(t)\vert+\frac{\vert(x-y)(t)\vert-\delta}{2\Lambda_{A_i}}\leq \vert x_i(t)\vert=1.$$
We have that
\begin{align*}
\sum_{i=1}^n\lambda_i y_i(t) &=\sum_{i\in A}\lambda_i x_i+\sum_{i\in B}\lambda_i x_i+(y-x)(t)\left(\frac{\sum_{i\in A}\lambda_i}{2\Lambda_A}+\frac{\sum_{i\in B}\lambda_i}{2\Lambda_B} \right)+ \delta\left(\frac{\sum_{i\in A}\lambda_i}{2\Lambda_A}-\frac{\sum_{i\in B}\lambda_i}{2\Lambda_B} \right)\\ & =\sum_{i=1}^n\lambda_i x_i(t)-z(t)+y(t)=y(t).
\end{align*}
Therefore, in that case we get
$$\vert y_i(t)-x_i(t)\vert\leq \frac{2\delta}{\min\limits_{1\leq j\leq n}\lambda_j}.$$

Summarising, we get that $y=\sum_{i=1}^n\lambda_i y_i$ for suitable $y_i\in B_{\ell_\infty(I)}$ for every $i\in\{1,\ldots, n\}$ satisfying that $\vert y_i(t)-x_i(t)\vert\leq \frac{2\delta}{3}$. Fix $i\in\{1,\ldots, n\}$ and let us show, to finish the proof, that $y_i\in S(B_{\ell_\infty(I)},f_i,\alpha)$. For this, we consider
\begin{align*}
f_i(y_i) &=\sum_{t\in I}f_i(t)y_i(t)=\sum_{t\in F}f_i(t)y_i(t)-\frac{\delta}{3} \\ & =\sum_{t\in F} f_i(t)x_i(t)+\sum_{t\in F} f_i(t)(y_i(t)-x_i(t))-\frac{\delta}{3} \\ & >1-\alpha+\delta
-\frac{2\delta}{3}\Vert f_i\Vert-\frac{\delta}{3}=1-\alpha
\end{align*}
since, from our estimates, $\vert x_i(t)-y_i(t)\vert<\frac{2\delta}{3}$ for every $i\in\{1,\ldots, n\}$ and every $t\in I$. This proves that $y_i\in S(B_{\ell_\infty(I)},f_i,\alpha)$, from where we deduce that $y\in C$ which finishes the proof.\end{proof}

\begin{proof}[Proof of Theorem \ref{teo:caraW*2linfinito}]
\eqref{teo:caraW*2linfinito1}$\Rightarrow$\eqref{teo:caraW*2linfinito2} is obvious, whereas \eqref{teo:caraW*2linfinito2}$\Rightarrow$\eqref{teo:caraW*2linfinito3} is Proposition \ref{propoLinftynow*2} and \eqref{teo:caraW*2linfinito3}$\Rightarrow$\eqref{teo:caraW*2linfinito1} is Proposition \ref{prop:condisufilinfty}.
\end{proof}

In order to get a consequence for $L_1$ preduals we will need the following proposition, which connects (W2) in a Banach space with the property (W$^*$2) in its bidual.

\begin{proposition}\label{subep2}
Let $X$ be a Banach space and assume that every convex combination of slices of $B_X$ has a weakly interior point. Then every convex combination of weak-star slices of $B_{X^{**}}$ contains some weak-star interior point. In other words, if $X$ has (W2), then $X^{**}$ has (W$^*$2).
\end{proposition}

\begin{proof}
Consider $C:=\sum_{i=1}^n\lambda_i S(B_{X^{**}},f_i,\alpha)$ to be a convex combination of weak-star slices in $B_{X^{**}}$. Pick $0<\delta<\alpha$ and define $D:=\sum_{i=1}^n \lambda_i S(B_X,f_i,\delta)$. By the assumption, we can find $x\in D$ and a weakly-star open subset $O$ of $X^{**}$ such that
$$x\in O\cap B_X\subseteq D.$$
Then
\begin{align*}
x\in O\cap B_{X^{**}} &\subseteq \overline{O\cap B_X}^{w^*}\subseteq \overline{D}^{w^*} =\sum_{i=1}^n\lambda_i \overline{S(B_{X^{**}},f_i,\delta)}^{w^*}\\
& =\sum_{i=1}^n \lambda_i \{x^{**}\in B_{X^{**}}:x^{**}(f_i)\geq 1-\delta\} \subseteq C,
\end{align*}
so $x\in C$ is a weakly-star interior point, as desired.
\end{proof}

In \cite[Theorem 3.1]{hkp} it is proved that $C(K)$ contains a convex combination of slices without any weak interior point whenever $K$ admits an atomeless measure. Note that this result can be seen as a part of the following more general result whose proof is an straightforward application of Proposition \ref{subep2} and Theorem \ref{teo:caraW*2linfinito}.

\begin{theorem}\label{condinecepredul1}
Let $X$ be a predual of $L_1$, that is, $X^*=L_1(\mu)$. If every convex combination of slices of $B_X$ contains some weak interior point then $\mu$ is purely atomic. In other words, if $X$ has (W2), then $\mu$ is purely atomic.
\end{theorem}

Let us end with a brief discussion about the weak and weak-star versions of the properties in dual Banach spaces. In general, the following diagram holds:
\begin{equation}\label{dia-weakstar}
\begin{tikzpicture}[baseline=(m.center)]
  \matrix (m) [matrix of math nodes,row sep=3em,column sep=4em,minimum width=2em]
  {
     \text{(W$^*$1)} & \text{(W$^*$2)} & \text{(W$^*$-CS)} \\
     \text{(W1)} & \text{(W2)} &  \text{(CS)}\\};
  \path[-stealth]
    (m-1-1) edge [double] node [above] {(1)} (m-1-2)
    (m-1-2) edge [double] node [above] {(2)} (m-1-3)
    (m-2-1) edge [double]  (m-2-2)
    (m-2-2) edge [double]  (m-2-3)
    (m-2-3) edge [double] node [right] {(3)} (m-1-3)
    (m-1-1) edge [double,neg] node [above] {(4)} (m-2-2)
    ;
\end{tikzpicture}
\end{equation}
The implications (1), (2), and (3) are obvious. Let us give an example showing that (W$^*$1) does not imply (W2) (this is (4)), and so showing that (W$^*$1) does not imply (W1) and (W$^*$2) does not imply (W2).

\begin{example}
$X=\ell_\infty$ has (W$^*$1) by Proposition \ref{prop:condisufilinfty}. However, from the identification $\ell_\infty=C(\beta\mathbb N)$, we deduce that $X$ fails $(W2)$ since $\beta\mathbb N$ is not scattered and we may use \cite[Remark 3.1]{hkp}.
\end{example}

Let us now present some examples showing that the implications (1), (2), and (3) in the diagram \eqref{dia-weakstar} do not reverse.

\begin{example} Let us consider the following examples.
\begin{itemize}
\item[(a)] $\ell_\infty\oplus_1\ell_\infty$ fails (W$^*$1) by a weak star version of \cite[Proposition 2.1]{hkp}. However, $c_0\oplus_1 c_0$ has (W2) by using \cite[Theorem 2.4]{al} and \cite[Proposition 2.4]{hkp}. Hence, $(c_0\oplus_1 c_0)^{**}=\ell_\infty\oplus_1 \ell_\infty$ has (W$^*$2) by Proposition \ref{subep2}. This shows that the reverse implication to (1) does not hold.
\item[(b)] $X=L_\infty[0,1]$ as dual of $L_1[0,1]$ has (W$^*$-CS) as it is the dual of a Banach space with the Daugavet property and we may use \cite[Example 3.4]{al}. However, $L_\infty[0,1]$ fails (W$^*$2) by Theorem \ref{teo:caraW*2linfinito}. This shows that (2) does not reverses.
\item[(c)] Let $X=L_1[0,1]^{**}$. Then $X$ fails (CS) since $B_X$ has strongly exposed points. However, $X$ has (W$^*$-CS) as it is the dual of a Banach space with the Daugavet property, $L_\infty[0,1]$, and we may use \cite[Example 3.4]{al}. This shows that the reverse implication to (3) does not hold.
\end{itemize}
\end{example}

\end{document}